\documentclass[12pt]{amsart}
\calclayout
\usepackage{amssymb}
\usepackage{amsmath}
\usepackage{graphics}
\usepackage{epsfig, color}
\newcommand\R{\mathbb R}
\newcommand\RR{\mathbb R}

\newcommand\D{\mathcal D}
\newcommand\K{\mathcal K}

\newcommand\T{\mathcal T}
\newcommand\E{\mathcal E}

\renewcommand\div{\operatorname{div}}

\newcommand\eps{\operatorname{\epsilon}}

\newcommand\x{\times}
\renewcommand\d{\mathrm d}
\renewcommand\t{\tilde}
\newcommand\lbra{[\![}
\newcommand\rbra{]\!]}
\newcommand\lbrac{\,[\!\!\!\{}
\newcommand\rbrac{\}\!\!\!]\,}

\renewcommand\ll{|\kern-2pt|\kern-2pt|}

\numberwithin{equation}{section}
\theoremstyle{plain}
\newtheorem{thm}{Theorem}

\newtheorem{lem}[thm]{Lemma}

\numberwithin{thm}{section}

\theoremstyle{remark}

\def\stack#1#2#3{\rlap{#1}\lower#3\hbox{#2}}

\def\twiddlespace{1.2truept}

\def\dtwiddle{\displaystyle\sim}
\def\ttwiddle{\textstyle\sim}
\def\stwiddle{\scriptstyle\sim}
\def\sstwiddle{\scriptscriptstyle\sim}

\def\doubledtwiddle{\stack{$\dtwiddle$}{$\dtwiddle$}{\twiddlespace}}
\def\doublettwiddle{\stack{$\ttwiddle$}{$\ttwiddle$}{\twiddlespace}}
\def\doublestwiddle{\stack{$\stwiddle$}{$\stwiddle$}{\twiddlespace}}
\def\doublesstwiddle{\stack{$\sstwiddle$}{$\sstwiddle$}{\twiddlespace}}

\def\tripledtwiddle{\stack{$\dtwiddle$}{$\doubledtwiddle$}{\twiddlespace}}
\def\triplettwiddle{\stack{$\ttwiddle$}{$\doublettwiddle$}{\twiddlespace}}
\def\triplestwiddle{\stack{$\stwiddle$}{$\doublestwiddle$}{\twiddlespace}}
\def\triplesstwiddle{\stack{$\sstwiddle$}{$\doublesstwiddle$}{\twiddlespace}}

\def\quadrupledtwiddle{\stack{$\dtwiddle$}{$\tripledtwiddle$}{\twiddlespace}}
\def\quadruplettwiddle{\stack{$\ttwiddle$}{$\triplettwiddle$}{\twiddlespace}}
\def\quadruplestwiddle{\stack{$\stwiddle$}{$\triplestwiddle$}{\twiddlespace}}
\def\quadruplesstwiddle{\stack{$\sstwiddle$}{$\triplesstwiddle$}{\twiddlespace}}\def\quadru
pletwiddle{{\mathchoice{\quadrupledtwiddle}{\quadruplettwiddle}%

{\quadruplestwiddle}{\quadruplesstwiddle}}}

\catcode`\!=11
\newcommand\!a{{\boldsymbol a}}
\newcommand\!b{{\boldsymbol b}}
\newcommand\!c{{\boldsymbol c}}
\newcommand\!d{{\boldsymbol d}}
\newcommand\!e{{\boldsymbol e}}
\newcommand\!f{{\boldsymbol f}}
\newcommand\!g{{\boldsymbol g}}
\newcommand\!h{{\boldsymbol h}}
\newcommand\!i{{\boldsymbol i}}
\newcommand\!j{{\boldsymbol j}}
\newcommand\!k{{\boldsymbol k}}
\newcommand\!l{{\boldsymbol l}}
\newcommand\!m{{\boldsymbol m}}
\newcommand\!n{{\boldsymbol n}}
\newcommand\!o{{\boldsymbol o}}
\newcommand\!p{{\boldsymbol p}}
\newcommand\!q{{\boldsymbol q}}
\newcommand\!r{{\boldsymbol r}}
\newcommand\!s{{\boldsymbol s}}
\newcommand\!t{{\boldsymbol t}}
\newcommand\!u{{\boldsymbol u}}
\newcommand\!v{{\boldsymbol v}}
\newcommand\!w{{\boldsymbol w}}
\newcommand\!x{{\boldsymbol x}}
\newcommand\!y{{\boldsymbol y}}
\newcommand\!z{{\boldsymbol z}}
\newcommand\!A{{\boldsymbol A}}
\newcommand\!B{{\boldsymbol B}}
\newcommand\!C{{\boldsymbol C}}
\newcommand\!D{{\boldsymbol D}}
\newcommand\!E{{\boldsymbol E}}
\newcommand\!F{{\boldsymbol F}}
\newcommand\!G{{\boldsymbol G}}
\newcommand\!H{{\boldsymbol H}}
\newcommand\!I{{\boldsymbol I}}
\newcommand\!J{{\boldsymbol J}}
\newcommand\!K{{\boldsymbol K}}
\newcommand\!L{{\boldsymbol L}}
\newcommand\!M{{\boldsymbol M}}
\newcommand\!N{{\boldsymbol N}}
\newcommand\!O{{\boldsymbol O}}
\newcommand\!P{{\boldsymbol P}}
\newcommand\!Q{{\boldsymbol Q}}
\newcommand\!R{{\boldsymbol R}}
\newcommand\!S{{\boldsymbol S}}
\newcommand\!T{{\boldsymbol T}}
\newcommand\!U{{\boldsymbol U}}
\newcommand\!V{{\boldsymbol V}}
\newcommand\!W{{\boldsymbol W}}
\newcommand\!X{{\boldsymbol X}}
\newcommand\!Y{{\boldsymbol Y}}
\newcommand\!Z{{\boldsymbol Z}}
\newcommand\!alpha{{\boldsymbol\alpha}}
\newcommand\!beta{{\boldsymbol\beta}}
\newcommand\!gamma{{\boldsymbol\gamma}}
\newcommand\!delta{{\boldsymbol\delta}}
\newcommand\!epsilon{{\boldsymbol\epsilon}}
\newcommand\!zeta{{\boldsymbol\zeta}}
\newcommand\!eta{{\boldsymbol\eta}}
\newcommand\!theta{{\boldsymbol\theta}}
\newcommand\!iota{{\boldsymbol\iota}}
\newcommand\!kappa{{\boldsymbol\kappa}}
\newcommand\!lambda{{\boldsymbol\lambda}}
\newcommand\!mu{{\boldsymbol\mu}}
\newcommand\!nu{{\boldsymbol\nu}}
\newcommand\!xi{{\boldsymbol\xi}}
\newcommand\!pi{{\boldsymbol\pi}}
\newcommand\!rho{{\boldsymbol\rho}}
\newcommand\!sigma{{\boldsymbol\sigma}}
\newcommand\!tau{{\boldsymbol\tau}}
\newcommand\!upsilon{{\boldsymbol\upsilon}}
\newcommand\!phi{{\boldsymbol\phi}}
\newcommand\!chi{{\boldsymbol\chi}}
\newcommand\!psi{{\boldsymbol\psi}}
\newcommand\!omega{{\boldsymbol\omega}}
\newcommand\!varepsilon{{\boldsymbol\varepsilon}}
\newcommand\!vartheta{{\boldsymbol\vartheta}}
\newcommand\!varpi{{\boldsymbol\varpi}}
\newcommand\!varrho{{\boldsymbol\varrho}}
\newcommand\!varsigma{{\boldsymbol\varsigma}}
\newcommand\!varphi{{\boldsymbol\varphi}}
\newcommand\!Gamma{{\boldsymbol\Gamma}}
\newcommand\!Delta{{\boldsymbol\Delta}}
\newcommand\!Theta{{\boldsymbol\Theta}}
\newcommand\!Lambda{{\boldsymbol\Lambda}}
\newcommand\!Xi{{\boldsymbol\Xi}}
\newcommand\!Pi{{\boldsymbol\Pi}}
\newcommand\!Sigma{{\boldsymbol\Omega\eigma}}
\newcommand\!Upsilon{{\boldsymbol\Upsilon}}
\newcommand\!Phi{{\boldsymbol\Phi}}
\newcommand\!Psi{{\boldsymbol\Psi}}
\newcommand\!Omega{{\boldsymbol\Omega}}

\begin{document}

\title 
{Discrete Korn's inequality for shells}

\author{Sheng Zhang}
\thanks{Department of Mathematics, Wayne State University, Detroit, MI 48202 (\texttt{szhang@wayne.edu})}


\begin{abstract}
We prove Korn's inequalities for Naghdi and Koiter shell models defined on spaces of discontinuous piecewise functions.
They are useful in study of discontinuous finite element methods for shells.
\vspace{12pt}

\end{abstract}
\maketitle

\section{Introduction}
In the Naghdi shell model, the strain energy is a sum of bending, membrane, and transverse shear strain energies.
The three strains are expressed in terms of the model primary variables comprising  displacement of the shell mid-surface and rotation of normal fibers,
in which derivatives and function values are combined
together in a complicated manner by the curvature tensors and Christoffel symbols.
Korn's  inequality for shells \cite{BCM}  establishes an equivalence between the strain energy norm of the primary variables 
and the their  usual Sobolev norms, which ensures  the wellposedness of the
the Naghdi shell model in Sobolev spaces.
The situation for the Koiter shell model is similar, which excludes the transverse shear effect in shell deformation and only uses 
mid-surface displacement as the primary variables, for which  a Korn's inequality is also available \cite{BCM}.
When analyzing conforming finite element methods for shells \cite{Arnold-Brezzi, Bernadou, Bathe-book}, 
such continuous version of Korn's inequalities play a fundamental role
in that
they attribute most of the numerical analysis and estimates to that in Sobolev spaces.
To deal with discontinuous Galerkin methods for shells \cite{Koiter-shell, Naghdi-shell},
it is desirable to have appropriate generalizations of the Korn's inequalities that  can be applied to piecewise defined discontinuous functions. 
It seems that such discrete
Korn's inequalities do not trivially follow from the Korn's inequalities for shells.
Methods of proving discrete Korn's inequality for plane elasticity \cite{Brenner-2} can not be easily adapted to shell problems either. 
It is the purpose of
this paper to establish such discrete Korn's inequalities for both the Naghdi and Koiter shell models, for the most general geometry of shell mid-surfaces,
without making any additional assumption on the regularity of shell model solutions.

Let $\t\Omega\subset\R^3$
be the middle surface of a shell of thickness $2\eps$.
It is  the image of a two-dimensional domain
$\Omega\subset\R^2$ through a mapping $\!Phi$ that has third order continuous derivatives. 
The Naghdi shell model is a two-dimensional model defined on $\Omega$.
In the following, we use super scripts to indicate contravariant components of tensors, and
use subscripts to indicate covariant components. Greek super and subscripts take their values in $\{1, 2\}$,
and Latin scripts take their values
in $\{1, 2, 3\}$. Summation rules with respect to repeated and super and subscripts will also be used.
We use $H^\alpha(\tau)$ to denote the $L^2$ based Sobolev space of order $\alpha$ of functions defined on $\tau$, in which  the 
norm and semi-norm are denoted by $\|\cdot\|_{k, \tau}$ and  $|\cdot|_{k, \tau}$, respectively. We also use $H^0(\tau)$ to denote $L^2(\tau)$. 
 When $\tau=\Omega$, we simply write the space as $H^\alpha$.

The coordinates $x_\alpha\in\Omega$
furnish the curvilinear coordinates on $\t\Omega$ through the mapping $\!Phi$.
We assume that at any point on the surface,
along the coordinate lines,
the two tangential vectors
$\!a_{\alpha}={\partial\!Phi}/{\partial x_{\alpha}}$
are linearly independent.
The unit vector
$\!a_3=(\!a_1\x\!a_2)/|\!a_1\x\!a_2|$ is normal to $\t\Omega$.
The triple $\!a_i$ furnishes the covariant basis on $\t\Omega$.
The contravariant basis
$\!a^i$ is defined by the relations
$\!a^{\alpha}\cdot\!a_{\beta}=\delta^{\alpha}_{\beta}$ and $\!a^3=\!a_3$,
in which $\delta^{\alpha}_{\beta}$ is the Kronecker delta.
It is obvious that $\!a^{\alpha}$ are also tangent to the surface.
The metric tensor has the covariant components
$a_{\alpha\beta}=\!a_{\alpha}\cdot\!a_{\beta}$,  the determinant of
which is denoted by $a$. The contravariant components
are given by
$a^{\alpha\beta}=\!a^{\alpha}\cdot\!a^{\beta}$.
The curvature tensor
has covariant components
$b_{\alpha\beta}=\!a_3\cdot\partial_{\beta}\!a_{\alpha}$, whose
mixed components are $b^{\alpha}_{\beta}=a^{\alpha\gamma}b_{\gamma\beta}$.
The symmetric tensor $c_{\alpha\beta}=b^\gamma_\alpha b_{\gamma\beta}$ is called the third
fundamental form of the surface.
The Christoffel symbols
are defined by
$\Gamma^{\gamma}_{\alpha\beta}
=\!a^{\gamma}\cdot\partial_{\beta}\!a_{\alpha}$,
which are symmetric with respect to the subscripts.

The Naghdi shell model \cite{Naghdi} uses the covariant components $u_\alpha$ of the 
tangential displacement vector $\!u=u_\alpha\!a^\alpha$, coefficient $w$ of the normal displacement $w\!a^3$ of the shell mid-surface,  and
the covariant components $\theta_\alpha$ of the  normal fiber rotation vector $\!theta=\theta_\alpha\!a^\alpha$ as the primary variables.
The bending strain tensor, membrane strain tensor, and transverse shear strain vector associated with a deformation 
represented by
such a set of primary variables are
\begin{equation}\label{N-bending}
\rho_{\alpha\beta}(\!theta, \!u, w)=
\frac12(\theta_{\alpha|\beta}+\theta_{\beta|\alpha})-\frac12(b^\gamma_\alpha u_{\gamma|\beta}+b^\gamma_\beta u_{\gamma|\alpha})+c_{\alpha\beta}w,
\end{equation}
\begin{equation}\label{N-metric}
\gamma_{\alpha\beta}(\!u,w)=
\frac12(u_{\alpha|\beta}+u_{\beta|\alpha})
-b_{\alpha\beta}w,
\end{equation}
\begin{equation}\label{N-shear}
\tau_\alpha(\!theta, \!u, w)=\partial_\alpha w+b^\gamma_\alpha u_\gamma+\theta_\alpha.
\end{equation}
Here, the covariant derivative of a vector $u_\alpha$ is defined by
\begin{equation}\label{covariant-derivative}
u_{\alpha|\beta}=\partial_{\beta}u_{\alpha}-\Gamma^{\gamma}_{\alpha\beta}
u_{\gamma}.
\end{equation}
The loading forces on the shell body and upper and lower surfaces
enter the shell model as resultant loading forces per unit area  on the shell middle surface,
of which the tangential force density is
$p^{\alpha}\!a_{\alpha}$ and transverse force density $p^3\!a_3$.
Let the boundary $\partial\t\Omega$ be divided to $\partial^D\t\Omega\cup\partial^S\t\Omega\cup\partial^F\t\Omega$.
On $\partial^D\t\Omega$ the shell is clamped, on $\partial^S\t\Omega$ the shell is soft-simply supported, and
on $\partial^F\t\Omega$ the shell is free of displacement constraint and subject to force or moment  only.
(There are $32$ different ways to specify boundary conditions at any point on the shell boundary, of which we consider the three most typical.)
Let $\!H^1=H^1\x H^1$. The shell model is
defined in the Hilbert space
\begin{multline}\label{N-space}
H=\{(\!phi, \!v, z)\in \!H^1\x\!H^1\x H^1;\  v_\alpha \text{ and } z \text{ are }0\  \text{on}\ \partial^D\Omega\cup\partial^S\Omega, \\
\text{ and  }\theta_\alpha \text{ is }0\
\text{on}\ \partial^D\Omega\}.
\end{multline}
The model determines $(\!theta, \!u, w)\in H$
such that
\begin{multline}\label{N-model}
\frac13\int_{\Omega}
a^{\alpha\beta\lambda\gamma}\rho_{\lambda\gamma}(\!theta, \!u, w)
\rho_{\alpha\beta}
(\!phi, \!v, z)\sqrt a\d x_1 \d x_2\\
+\eps^{-2}\left[\int_{\Omega}
a^{\alpha\beta\lambda\gamma}\gamma_{\lambda\gamma}(\!u,w)
\gamma_{\alpha\beta}(\!v,z)+\kappa\mu\int_{\Omega}a^{\alpha\beta}\tau_\alpha(\!theta, \!u, w)\tau_\beta(\!phi, \!v, z)\right]\sqrt a\d x_1\d x_2
\\
=
\int_{\Omega}
(p^{\alpha}v_{\alpha}+
p^3z)\sqrt a\d x_1\d x_2
+\int_{\partial^S\t\Omega}r^\alpha\phi_\alpha\d\t s
+\int_{\partial^F\t\Omega}\left(q^\alpha v_\alpha+q^3z+r^\alpha\phi_\alpha\right)\d\t s
\\ \forall\
(\!phi, \!v,z) \in H.
\end{multline}
Here, $q^i$ and $r^\alpha$ are the force resultant and
moment resultant on the shell edge \cite{Naghdi}.
The factor $\kappa$ is a shear correction factor.
The last two integrals on the shell edge is taken with respect to the arc length of the boundary of $\t\Omega$. 
The fourth order contravariant tensor
$a^{\alpha\beta\gamma\delta}$ is the elastic tensor of the shell,
defined by
\begin{equation*}
a^{\alpha\beta\gamma\delta}=\mu (a^{\alpha\gamma}a^{\beta\delta}+a^{\beta\gamma}a^{\alpha\delta})+
\frac{2\mu\lambda}{2\mu+\lambda}
a^{\alpha\beta}a^{\gamma\delta}.
\end{equation*}
Here, $\lambda$ and $\mu$ are the Lam\'e coefficients of the elastic material. It satisfies the condition that 
there are constants $C_1$ and $C_2$
that only depend on the shell mid-surface and the Lam\'e coefficients of the shell material such that for any
tensor $\varsigma_{\alpha\beta}$
\begin{equation*}
\sum_{\alpha,\beta=1}^2|\varsigma_{\alpha\beta}|^2\le C_1a^{\alpha\beta\lambda\gamma}\varsigma_{\alpha\beta}\varsigma_{\lambda\gamma},\quad
a^{\alpha\beta\lambda\gamma}\varsigma_{\alpha\beta}\varsigma_{\lambda\gamma}\le C_2\sum_{\alpha,\beta=1}^2|\varsigma_{\alpha\beta}|^2.
\end{equation*}
Also, there are constants $C_1$ and $C_2$ that only depend on the shell mid-surface such that for any vector $\varsigma_\alpha$, 
\begin{equation*}
\sum_{\alpha=1}^2|\varsigma_\alpha|^2\le C_1 a^{\alpha\beta}\varsigma_\alpha \varsigma_\beta,\quad
a^{\alpha\beta}\varsigma_\alpha \varsigma_\beta\le C_2\sum_{\alpha=1}^2|\varsigma_\alpha|^2.
\end{equation*}
These inequalities together with a Korn's inequality \cite{BCM} assures that  
the Naghdi shell model  \eqref{N-model} has a unique solution in the space $H$. The Korn's inequality states that  
there is a $C$ that could be dependent on the shell mid-surface such that 
\begin{multline}\label{N-Korn-continuous}
\|\!theta\|_{\!H^1}+\|\!u\|_{\!H^1}+\|w\|_{H^1}\\
\le C\left[
\sum_{\alpha,\beta=1}^2\|\rho_{\alpha\beta}(\!theta, \!u, w)\|^2_{L^2}+
\sum_{\alpha,\beta=1}^2\|\gamma_{\alpha\beta}(\!u, w)\|^2_{L^2}+
\sum_{\alpha=1}^2\|\tau_{\alpha}(\!theta, \!u, w)\|^2_{L^2}+f^2(\!theta, \!u, w)\right]^{1/2}\\
\forall\ (\!theta,\!u, w)\in \!H^1\x\!H^1\x H^1.
\end{multline}
Here $f$ is a continuous  seminorm satisfying the rigid body motion condition that if  $(\!theta, \!u, w)$ defines a rigid body motion and 
$f(\!theta, \!u, w)=0$ then $(\!theta, \!u, w)=0$. 
The displacement functions  $(\!theta, \!u, w)\in \!H^1\x\!H^1\x H^1$ defines a rigid body if and only if
$\rho_{\alpha\beta}(\!theta, \!u, w)=0$,  $\gamma_{\alpha\beta}(\!u, w)=0$,  and $\tau_{\alpha}(\!theta, \!u, w)=0$ \cite{BCM}.

We assume that $\Omega$ is a bounded polygon.
Let $\T_h$ be a shape regular, but not necessarily quasi-uniform,
triangulation on $\Omega$.
Shape regularity of triangulations
is a crucial notion in this paper. It is worthwhile to recall its definition
here. Considering a triangle, we let $r$ and $R$ be the radii of its
inscribed circle and circumcircle, respectively.
Then the ratio $R/r$ is called its shape regularity constant, or simply shape regularity.
For a triangulation, the maximum
of shape regularities of all its triangles
is called the shape regularity of the triangulation \cite{Ciarlet-FEM-book}, denoted by $\K$.
We will need to consider a (infinite) class of triangulations. For a class, the
{\em shape regularity} $\K$ is the supreme of all the shape regularities
of its triangulations.
For the triangulation $\T_h$, we use $\T_h$ to denote the set of all (open) triangular elements
of the partition,
and use $\Omega_h$ to denote the union of
all the open triangular elements.
We use $\E^0_h$ to denote the set of all interior (open) edges and $\E^\partial_h$ all
boundary edges, and let $\E_h=\E^0_h\cup\E^\partial_h$. We use $h_\tau$ to denote the diameter of an element 
$\tau\in\T_h$ and use
$h_e$ to denote the length of an edge $e\in \E_h$.

Let $H^1_h$ be the space of piecewise $H^1$ functions subordinated to the triangulation
$\T_h$.
A function in $H^1_h$ is independently defined on
every element $\tau\in\T_h$ on which it belongs to $H^1(\tau)$. 
A function $u$ in $H^1_h$
is certainly
in $L^2(\Omega_h)$. 
On an edge $e\in\E^0_h$, a function $u$ may have two different traces from the
two elements sharing $e$. We use $\lbra u\rbra$ to denoted the absolute value of difference
of the two traces, which is the jump of $u$ over $e$.
In the space
$H^1_h$, we define a norm
\begin{equation}\label{H1hnorm}
\|u\|_{H^1_h}:=\left[\sum_{\tau\in\T_h}\|u\|^2_{1,\tau}
+\sum_{e\in\E^0_h}\frac{1}{h_e}\int_e\lbra u\rbra^2\d s\right]^{1/2}.
\end{equation}
Let $\!H^1_h=H^1_h\x H^1_h$. 
We prove the discrete Korn's inequality for Naghdi shell that 
for all $(\!theta,\!u, w)\in \!H^1_h\x\!H^1_h\x H^1_h$
\begin{multline}\label{N-Korn-discrete}
\|\!theta\|^2_{\!H^1_h}+\|\!u\|^2_{\!H^1_h}+\|w\|^2_{H^1_h}\\
\le C\left[
\sum_{\alpha,\beta=1}^2\|\rho_{\alpha\beta}(\!theta, \!u, w)\|^2_{0, \Omega_h}+
\sum_{\alpha,\beta=1}^2\|\gamma_{\alpha\beta}(\!u, w)\|^2_{0, \Omega_h}+
\sum_{\alpha=1}^2\|\tau_{\alpha}(\!theta, \!u, w)\|^2_{0, \Omega_h}+f^2(\!theta, \!u, w)\right.\\
\left. +\sum_{e\in\E^0_h}
\frac{1}{h_e}\int_e
\sum_{\alpha=1}^2\left(\lbra \theta_\alpha\rbra^2+\lbra u_\alpha\rbra^2\right)\d s
+\sum_{e\in\E^0_h}\frac{1}{h_e}\int_e\lbra w\rbra^2\d s\right].
\end{multline}
Here $f$ is a seminorm such that $f(\!theta, \!u, w)\le C(\|\!theta\|_{\!H^1_h}+\|\!u\|_{\!H^1_h}+\|w\|_{H^1_h})$
and it satisfies the rigid body motion condition that 
if  $(\!theta, \!u, w)$ defines a rigid body motion and $f(\!theta, \!u, w)=0$ then $(\!theta, \!u, w)=0$.
The constant $C$ could be dependent on the shell mid-surface and the shape regularity $\K$ of the triangulation
$\T_h$, but otherwise is independent of the triangulation. We shall simply say that such a constant is independent of $\T_h$.

The Koiter shell model \cite{Koiter} uses the covariant components $u_\alpha$ of the tangential displacement $u_\alpha\!a^\alpha$
and the coefficient $w$ of the normal  displacement $w\!a^3$  
of the shell mid-surface as the primary variable. Such a displacement 
deforms the surface $\t\Omega$ and changes its curvature and metric tensors. The linearized change in curvature tensor is the bending
strain tensor. It is expressed in terms of the displacement components as 
\begin{equation}\label{K-curvature}
\rho^K_{\alpha\beta}(\!u, w)=\partial^2_{\alpha\beta}w-\Gamma^{\gamma}_{\alpha\beta}
\partial_{\gamma}w+b^{\gamma}_{\alpha|\beta}u_{\gamma}+
b^{\gamma}_{\alpha}u_{\gamma|\beta}+b^{\gamma}_{\beta}u_{\gamma|\alpha}-c_{\alpha\beta}w.
\end{equation}
The linearized change of metric tensor is the membrane strain tensor $\gamma_{\alpha\beta}(\!u,w)$ that has the same expression as in Naghdi model
\eqref{N-metric}.
There is no transverse shear in the Koiter model. Indeed the bending strain tensor \eqref{K-curvature} can be obtained from the bending strain 
tensor of the Naghdi model \eqref{N-bending} by eliminating the rotation vector $\!theta$ using the zero-shear condition 
that $\tau_\alpha(\!theta, \!u, w)=0$, and multiplying the result by $-1$. I.e.,
\begin{equation*}
\rho^K_{\alpha\beta}(\!u, w)=-\rho_{\alpha\beta}(\!theta, \!u, w) \text{ with } \theta_\alpha=-\partial_\alpha w-b^\beta_\alpha u_\beta.
\end{equation*}

As in the Naghdi model, 
we let the boundary $\partial\t\Omega$ be divided to $\partial^D\t\Omega\cup\partial^S\t\Omega\cup\partial^F\t\Omega$.
On $\partial^D\t\Omega$ the shell is clamped, on $\partial^S\t\Omega$ the shell is simply supported, and
on $\partial^F\t\Omega$ the shell is free of displacement constraint and subject to force or moment  only.
(There are $16$ different ways to specify boundary conditions at any point on the shell boundary, of which we consider the three most
typical.)
The shell model is
defined in the Hilbert space
\begin{multline}\label{K-space}
H^K=\{(\!v, z)\in \!H^1\x H^2\ | \ v_\alpha \text{ and } z \text{ are }0\  \text{on}\ \partial^D\Omega\cup\partial^S\Omega, \\
\text{ and the normal derivative of  }z\text{ is }0\
\text{on}\ \partial^D\Omega\}.
\end{multline}
The model determines $(\!u, w)\in H^K$
such that
\begin{multline}\label{K-model}
\frac13\int_{\Omega}
a^{\alpha\beta\lambda\gamma}\rho^K_{\lambda\gamma}(\!u, w)
\rho^K_{\alpha\beta}
(\!v, z)\sqrt a\d x_1\d x_2
+\eps^{-2}\int_{\Omega}
a^{\alpha\beta\lambda\gamma}\gamma_{\lambda\gamma}(\!u,w)
\gamma_{\alpha\beta}(\!v,z)\sqrt a\d x_1\d x_2\\
=
\int_{\Omega}
(p^{\alpha}v_{\alpha}+
p^3z)\sqrt a\d x_1\d x_2
+\int_{\partial^S\t\Omega}mD_{\t\!n}z\d\t s
+\int_{\partial^F\t\Omega}\left(q^\alpha v_\alpha+q^3z+mD_{\t\!n}z\right)\d\t s
\\ 
\forall\ (\!v,z) \in H^K.
\end{multline}
Here, $q^i$ and $m$ are resultant loading functions on the shell boundary, which can be calculated from force resultants and
moment resultants on the shell edge \cite{Koiter}. The scalar $z$ can be viewed as
defined on $\t\Omega$. We let $\t\!n=\t n^\alpha\!a_\alpha$ be the unit outward normal to $\partial\t\Omega$ that is tangent to
$\t\Omega$. The derivative $D_{\t \!n}z:=\t n^\alpha\partial_\alpha z$ is the directional derivative  in the direction of $\t\!n$ with respect to arc length.
The elastic tensor 
$a^{\alpha\beta\gamma\delta}$ is the same as in Naghdi model.
The wellposedness of the Koiter model \eqref{K-model}
hinges on a Korn's inequality for Koiter shell \cite{BCM} that  there is a constant $C$ such that
\begin{multline}\label{K-Korn-continuous}
\|\!u\|_{\!H^1}+\|w\|_{H^2}\le C
\left[\sum_{\alpha,\beta=1,2}\|\rho^K_{\alpha\beta}(\!u, w)\|^2_{L^2}+\sum_{\alpha,\beta=1,2}\|\gamma_{\alpha\beta}(\!u, w)\|^2_{L^2}+f^2(\!u, w)\right]^{1/2}
\\
\ \ \forall\
\!u\in \!H^1,\ w\in H^2.
\end{multline}
Here $f(\!u, w)$ is a semi-norm on $\!H^1\x H^2$ that satisfies the rigid body motion condition that
if $(\!u, w)$ defines a rigid body motion and $f(\!u, w)=0$ then $\!u=0$ and $w=0$.
A displacement $(\!u, w)\in \!H^1\x H^2$ defines  a rigid body motion of the shell mid-surface
if there are constant vectors $\!c$ and $\!d$ such that  $u_\alpha\!a^\alpha+w\!a^3=\!c+\!d\x\!Phi(x_1, x_2)$. 
This is equivalent to $\rho^K_{\alpha\beta}(\!u, w)=0$ and  $\gamma_{\alpha\beta}(\!u, w)=0$ \cite{BCM}.

For the Koiter model, on the triangulation $\T_h$, we also need to consider piecewise $H^2$ functions that are independently defined on each element, with 
the norm defined by
\begin{multline}\label{H2hnorm}
\|w\|_{H^2_h}:=
\left[\sum_{\tau\in\T_h}\|w\|^2_{2,\tau}
+\sum_{e\in \E^0_h}\left(\sum_{\alpha=1, 2}
h^{-1}_e\int_{e}\lbra \partial_\alpha w\rbra^2\d s
+
h^{-1}_e\int_{e}\lbra w\rbra^2\d s\right)\right]^{1/2}.\hfill
\end{multline}
Let $f(\!u, w)$ be a semi-norm that is continuous with respect to this norm such that there is a $C$ only dependent on $\K$ of $\T_h$
and
\begin{equation*} 
|f(\!u, w)|\le C(\|\!u\|_{\!H^1_h}+\|w\|_{H^2_h})\ \forall\ (\!u, w)\in \!H^1_h\x H^2_h.
\end{equation*}
We  assume that $f$ satisfies the rigid body motion condition that
if $(\!u, w)\in \!H^1\x H^2$ defines a rigid body motion and $f(\!u, w)=0$ then $\!u=0$ and $w=0$.
We have the discrete Korn's inequality for Koiter shell that for all $\!u\in \!H^1_h$ and $w\in  H^2_h$
\begin{multline}\label{K-Korn-discrete}
\|\!u\|_{\!H^1_h}+\|w\|_{H^2_h}\le C
\left[
\sum_{\alpha,\beta=1,2}\|\rho^K_{\alpha\beta}(\!u, w)\|^2_{0,\Omega_h}+\sum_{\alpha,\beta=1,2}\|\gamma_{\alpha\beta}(\!u, w)\|^2_{0,\Omega_h}\right.
\\
\left.+\sum_{e\in \E^0_h}\left(
\sum_{\alpha=1,2}h^{-1}_e\int_{e}\lbra u_{\alpha}\rbra^2\d s
+\sum_{\alpha=1, 2}
h^{-1}_e\int_{e}\lbra \partial_\alpha w\rbra^2\d s
+
h^{-1}_e\int_{e}\lbra w\rbra^2\d s\right)
+f^2(\!v, z)
\right]^{1/2}.
\end{multline}

With this inequality established, one may add the term $\sum_{e\in \E^0_h}h^{-3}_e\int_{e}\lbra w\rbra^2\d s$ to both sides to obtain a new inequality. 
Let $\lbra \partial_s w\rbra$ and $\lbra \partial_n w\rbra$ be the jumps in the tangential and normal derivatives of $w$ over and edge $e$, respectively. Then 
we have the identity
\begin{equation*}
\sum_{\alpha=1}^2\lbra \partial_\alpha w\rbra^2=\lbra \partial_s w\rbra^2+\lbra \partial_n w\rbra^2.
\end{equation*}
If $w$ is a piecewise polynomial, we have the inverse inequality $\int_e\lbra \partial_s w\rbra^2\le Ch^{-2}_e\int_e\lbra w\rbra^2$.
Thus for piecewise polynomials $u_\alpha$ and 
$w$, we have the following variant of discrete Korn's inequality for Koiter shell.
\begin{multline}\label{K-Korn-discrete-h-3}
\|\!u\|_{\!H^1_h}+\|w\|_{\bar H^2_h}\le C
\left[
\sum_{\alpha,\beta=1,2}\|\rho^K_{\alpha\beta}(\!u, w)\|^2_{0,\Omega_h}+\sum_{\alpha,\beta=1,2}\|\gamma_{\alpha\beta}(\!u, w)\|^2_{0,\Omega_h}\right.
\\
\left.+\sum_{e\in \E^0_h}\left(
\sum_{\alpha=1,2}h^{-1}_e\int_{e}\lbra u_{\alpha}\rbra^2\d s
+
h^{-1}_e\int_{e}\lbra \partial_nw\rbra^2\d s
+
h^{-3}_e\int_{e}\lbra w\rbra^2\d s\right)
+f^2(\!u, w)
\right]^{1/2}.
\end{multline}
Here 
\begin{multline}\label{H2hnorm-h-3}
\|w\|_{\bar H^2_h}:=
\left[\sum_{\tau\in\T_h}\|w\|^2_{2,\tau}
+\sum_{e\in \E^0_h}\left(
h^{-1}_e\int_{e}\lbra \partial_n w\rbra^2\d s
+
h^{-3}_e\int_{e}\lbra w\rbra^2\d s\right)\right]^{1/2}.\hfill
\end{multline}
This inequality is useful for analysis of discontinuous Galerkin methods for Koiter shell.


In proving these discrete Korn's inequalities, a important tool is a compact embedding theorem 
in the space $H^1_h$, 
which is proved in Section~\ref{sec-compact}.
In Section~\ref{Korn-Naghdi} we prove the discrete Korn's inequalities \eqref{N-Korn-discrete} 
for the Naghdi shell, and 
Section~\ref{Korn-Koiter} is devoted Koiter shell model.
Throughout the paper, $C$ will be a generic constant that may depend on the domain $\Omega$, the mapping $\!Phi$ that defines the shell mid-surface, and shape
regularity $\K$ of a triangle, of a triangulation, or of a class of triangulations. But otherwise,
the constant is independent of triangulations. An integral $\int_\tau u(x_1, x_2)\d x_1 \d x_2$ or $\int_{\partial\Omega}v(s)\d s$ 
will be simply written as $\int_\tau u$ or $\int_{\partial\Omega}v$, in which
the integration variable and measure should be clear from the context.

\section{Compact embedding in the space of piecewise $H^1$ functions}
\label{sec-compact}
We need a discrete analogue of the  Rellich–-Kondrachov compact embedding theorem, which will play a fundamental role 
in proving the discrete Korn's inequalities for shells. 
There are several papers relevant to compact embedding in piecewise function spaces. 
In \cite{Feng}, such a result is stated under  the assumption that the triangulation $\T_h$ is quasi-uniform.
In \cite{Ern}, a theorem was proved for piecewise polynomials under the assumption that the maximum mesh size tends to zero.
In \cite{Buffa}, there is a sub-mesh condition to be verified.
Although these theories are developed in more general settings, their results do not readily meet our needs.
We only assume that $\T_h$ is a shape regular triangulation of the polygon $\Omega$. Our proof also clearly 
shows that the statement could break down when an interior angle of $\Omega$ tends to zero or $2\pi$.

\subsection{A trace theorem}\label{subsec-trace}
We first prove a trace theorem for functions in $H^1_h$.
This result itself is a generalization of a trace theorem of Sobolev space theory.
It will be used in proving the discrete compact embedding theorem, and in verifying the continuity 
of the seminorm denoted by $f$ in the right hand side of \eqref{N-Korn-discrete}, \eqref{K-Korn-discrete}, and
 \eqref{K-Korn-discrete-h-3}.
We will need the following trace theorem on an element \cite{A-DG}.
\begin{lem}
Let $\tau$ be a triangle, and $e$ one of its edges. Then there is a
constant $C$ depending on the shape regularity of $\tau$ such that
\begin{equation}\label{trace}
\int_eu^2\le C\left[h_e^{-1}\int_{\tau}u^2+h_e\sum_{\alpha=1}^2\int_{\tau}|\partial_\alpha u|^2\right]
\ \ \forall\ u\in H^1(\tau).
\end{equation}
\end{lem}
\begin{thm}\label{tracetheorem}
Let $\T_h$ be a shape regular, but not necessarily quasi-uniform triangulation of $\Omega$. There
exists a constant $C$ such that
\begin{equation}\label{Omega-trace}
\|u\|_{L^2(\partial\Omega)}\le C \|u\|_{H^1_h}\ \ \forall\ u\in H^1_h.
\end{equation}
\end{thm}
\begin{proof}
Let $\!phi$ be a piecewise smooth
vector field on $\Omega$ whose normal component is continuous across any straight line segment, and
such that $\!phi\cdot\!n=1$ on $\partial\Omega$. 
On each element
$\tau\in\T_h$, we have
\begin{equation*}
\int_{\partial\tau}u^2\!phi\cdot\!n=\int_\tau \div(u^2\!phi)=
\int_\tau(2u\nabla u\cdot\!phi+u^2\div\!phi).
\end{equation*}
Here $\nabla u$ is the gradient of $u$. Summing up the above equations over all elements of $\T_h$, we get
\begin{equation*}
\int_{\partial\Omega}u^2=-\sum_{e\in\E^0_h}\int_e\lbra u^2\!phi\rbra
+\int_{\Omega_h}(2u\nabla u\cdot\!phi+u^2\div\!phi).
\end{equation*}
If $e$ is the border between the elements
$\tau_1$ and
$\tau_2$ with outward normals $\!n_1$ and $\!n_2$, then
$\lbra u^2\!phi\rbra=u^2_1\!phi_1\cdot\!n_1+u^2_2\!phi_2\cdot\!n_2$,
where $u_1$ and $u_2$ are restrictions of $u$
on $\tau_1$ and $\tau_2$, respectively.  It is noted that although $\!phi$ may be discontinuous
across $e$, it normal component is continuous, i.e., $\!phi_1\cdot\!n_1+\!phi_2\cdot\!n_2=0$.
On the edge $e$, we have
$|\lbra u^2\!phi\rbra|\le|\lbra u^2\rbra|\|\!phi\|_{0,\infty,\Omega}$.
Here, $|\lbra u^2\rbra|=|u_1^2-u_2^2|=2|\lbra u\rbra\lbrac u\rbrac|$,
with $\lbrac u\rbrac=(u_1+u_2)/2$ being the average. We have
\begin{multline}\label{jump-on-e}
\int_e|\lbra u^2\!phi\rbra|\le
2|\!phi|_{0,\infty,\Omega}
\left[h_e^{-1}\int_e\lbra u\rbra^2\right]^{1/2}\left[h_e\int_e\lbrac u\rbrac^2\right]^{1/2}\\
\le C
|\!phi|_{0,\infty,\Omega}\left[h_e^{-1}\int_e\lbra u\rbra^2\right]^{1/2}
\left[\int_{\delta e}u^2+h_e^2\int_{\delta e}|\nabla u|^2\right]^{1/2}.
\end{multline}
Here, $|\!phi|_{0,\infty,\Omega}$ is the Sobolev norm in the space $[W^{0,\infty}(\Omega)]^2$ and $C$ only depends on the shape regularity of $\tau_1$ and $\tau_2$.
We used $\delta e=\tau_1\cup\tau_2$ to denote the ``co-boundary''
of edge $e$,
and we used the trace
estimate \eqref{trace}.
It then follows from the Cauchy--Schwarz inequality that
\begin{equation*}
\|u\|^2_{L^2(\partial\Omega)}\le C(|\!phi|_{0,\infty,\Omega}+|\div\!phi|_{0,\infty,\Omega})
\left[\|u\|^2_{L^2}+\int_{\Omega_h}|\nabla u|^2
+\sum_{e\in\E^0_h}\frac{1}{|e|}\int_e\lbra u\rbra^2\right].
\end{equation*}
Here the constant $C$ only depends on the shape regularity of $\T_h$.
The dependence on $\Omega$ of the $C$ in \eqref{Omega-trace}
is hidden in the $\!phi$ in the above inequality.
\end{proof}
%
%

\subsection{Compact embedding in $H^1_h$}
\label{subsec-compact}
For $\delta>0$, we define a boundary strip $\Omega_\delta$ of width $\mathcal{O}(\delta)$ for the domain $\Omega$.
Let $\Omega_\delta^0=\Omega\setminus \overline{\Omega_\delta}$ be the interior domain.
The interior domain $\Omega^0_\delta$ has the property that if a point is in $\Omega^0_\delta$, then the
disk centered at the point with radius $\delta$ entirely lies in $\Omega$.
We first show that when the strip is thin, the $L^2(\Omega_\delta)$
norm of the restriction on $\Omega_\delta$ of a function in $H^1_h$ must be small.
\begin{lem}\label{bd-shift-theorem}
There is a constant $C$ such that
\begin{equation}\label{stripe}
\int_{\Omega_{\delta}}u^2\le C\delta\|u\|^2_{H^1_h}\ \ \forall\ u\in H^1_h.
\end{equation}
Here $\Omega_\delta$ is a boundary strip of width $\delta$ attached to $\partial\Omega$.
\end{lem}
\begin{proof}
We choose a piecewise smooth a vector field $\!phi$ whose normal component is continuous
across any curve such that $\!phi=0$ on the inner boundary of $\Omega_{\delta}$, and
$\div\!phi=1$ and $|\!phi|\le C\delta$ on $\Omega_{\delta}$.
(A construction of such $\!phi$ is given below.)
We then extend $\!phi$ by zero onto the entire domain $\Omega$. The extended, still denoted by $\!phi$,
is a piecewise smooth vector field whose normal components
is continuous over any curve in $\Omega$.
We thus have
\begin{equation*} 
\int_{\Omega_{\delta}}u^2=
\int_{\Omega}u^2\div\!phi
=\sum_{\tau\in\T_h}\int_{\tau}u^2\div\!phi=
-\sum_{\tau\in\T_h}\int_{\tau}2u\nabla u\cdot\!phi+
\sum_{\tau\in\T_h}\int_{\partial\tau}u^2\!phi\cdot\!n.
\end{equation*}
The last term can be written as
\begin{equation*}
\sum_{\tau\in\T_h}\int_{\partial\tau}u^2\!phi\cdot\!n=
\int_{\partial\Omega}u^2\!phi\cdot\!n+\sum_{e\in\E^0_h}\int_e\lbra u^2\!phi\rbra.
\end{equation*}
Since the normal components of $\!phi$ is continuous
on edges in $\E^0_h$, we use the same argument as in the proof of Theorem~\ref{tracetheorem}, cf., \eqref{jump-on-e},
to get
\begin{equation*}
\sum_{e\in\E^0_h}\int_e|\lbra u^2\!phi\rbra|\le C
|\!phi|_{0,\infty,\Omega}\left[\|u\|^2_{L^2}+\sum_{\tau\in\T_h}\int_{\tau}h^2_\tau|\nabla u|^2
+\sum_{e\in\E^0_h}\frac{1}{h_e}\int_e\lbra u\rbra^2\right].
\end{equation*}
It then follows from Theorem~\ref{tracetheorem} that
\begin{equation} \label{bd-shift}
\int_{\Omega_{\delta}}u^2\le C
|\!phi|_{0,\infty,\Omega}\left[\|u\|^2_{L^2}+\int_{\Omega_h}|\nabla u|^2
+\sum_{e\in\E^0_h}\frac{1}{h_e}\int_e\lbra u\rbra^2\right].
\end{equation}
The proof is complete since $|\!phi|_{0,\infty, \Omega}\le C\delta$.
The constant $C$ depends on $\Omega$ in terms of its interior angles and exterior angles at convex and
concave vertexes, respectively.
\end{proof}

We describe a way to choose the boundary strip and construct the vector field $\!phi$ that was
used in the proof.
This field can be constructed by piecing together several special vector fields.
We need some vector fields on rectangles, wedges, and circular disks.
On the $xy$-plane, we consider the vertical rectangular strip $R=(0,\delta)\x (0, l)$.
On this strip, we consider $\!psi_R=\langle x, 0\rangle$. This vector
field satisfies the condition that $\div\!psi_R=1$, $\!psi_R=0$ on the left side,
$\!psi_R\cdot\!n=0$ on the top and bottom sides and the maximum of $|\!psi_R|$ is $\delta$ that is attained
on the right side.
On a wedge $W$ with its vertex at the origin, we consider the vector field
$\!psi_W=\langle x, y\rangle/2$.
This $\!psi_W$ satisfies the conditions that
$\div\!psi_W=1$, $\!psi_W\cdot\!n=0$ on the two sides of $W$, and $|\!psi_W|=\rho/2$
at a point in $W$ whose distant from the origin is $\rho$.
On a circular disk $C$ centered at the origin and of radius $\rho$, we consider the vector field
$\!psi_C=(1-\rho^2/r^2)\langle x, y\rangle/2$. Here $r=(x^2+y^2)^{1/2}$.
This vector field satisfies the condition that $\div\!psi_C=1$ on the disk except at the center
where it is singular. It points toward the center  every where. And it is zero on the boundary.
We use this field on a sector of the circle $C$, on the two radial sides of which
we have $\!psi_C\cdot\!n=0$.
\begin{figure}[!ht]
\centerline{\input{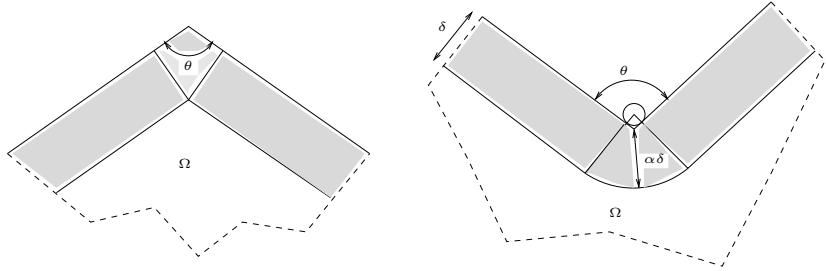}}
\caption{\label{boundarystrip}Boundary strip $\Omega_\delta$ near a convex vertex (left)
and a concave vertex (right).}
\end{figure}
With these special vector fields, we can then assemble the $\!phi$ on a boundary strip
$\Omega_\delta$. Along the interior side
of each straight segment of
$\partial\Omega$ we choose a uniform strip of thickness $\delta$. These strips overlap near
vertexes. If $\Omega$ is convex at a vertex, we introduce a wedge whose vertex is
at the intersection of the interior boundary of the uniform strips, and whose sides are orthogonal
to the meeting straight segments, see the left figure of Figure~\ref{boundarystrip}.
If $\Omega$ is concave at a vertex, we resolve it by using a circular sector, centered near the vertex
and outside of $\Omega$. The radius of the circle is slightly bigger than $\delta$
such that the arc is continuously connected to the
interior edges of the meeting strips, and the two radial sides are orthogonal to the meeting boundary
segments, see the right figure in Figure~\ref{boundarystrip}.
With such treatment of the vertices, the boundary strip $\Omega_\delta$
is composed of
rectangular strips attaching to major portions of straight segments of
$\partial\Omega$, portion of wedges at convex vertexes, and portion of circular sectors
at concave vertexes, see the shaded region in Figure~\ref{boundarystrip}.
We then transform $\!psi_R$, $\!psi_W$, and $\!psi_C$ to various parts of $\Omega_\delta$
and assemble a $\!phi$
on $\Omega_\delta$. The vector field $\!phi$ thus constructed is zero on the interior boundary
of $\Omega_\delta$. Its normal components are continuous across any curve, and
$\div\!phi=1$ on $\Omega_\delta$.
The thickness of $\Omega_\delta$ is the constant $\delta$ for the rectangular part.
It is maximized to $\delta/\sin\frac{\theta}2$ at a convex vertex. It is minimized to $\alpha\delta$
at the concave vertex, with $0<\alpha<1$, a value can be chosen as, for example, $1/2$. The norm $|\!phi|$
has a maximum $\delta/\sin\frac{\theta}2$ at a convex vertex with $\theta$ being the interior angle.
Thus when $\theta$ is small, $|\!phi|$ is big, and the estimate
\eqref{bd-shift} deteriorates.
The norm $|\!phi|$ also has a local maximum near a concave vertex.
It is bounded as
\begin{equation*}
|\!phi|\le\delta\frac{1-\alpha\sin\frac{\theta}{2}}{(1-\alpha)\sin\frac{\theta}{2}}.
\end{equation*}
When the exterior angle is small this maximum would be  big.
Also, one needs to choose $\alpha$ away from $1$ and $0$,
to maintain a moderate thickness of the strip and a reasonable bound for $|\!phi|$
which affect the estimate \eqref{bd-shift}.
We remark that the constant $C$ in the estimate \eqref{stripe} could tend to infinity if 
an interior angle of the polygonal domain $\Omega$ tends to zero or $2\pi$.

We then prove that functions in $H^1_h$ are ``shift-continuous'' in $L^2$, as
stated in the following lemma. We extend a function $u\in H^1_h$
to a function $\tilde u$ on the whole $\RR^2$ by zero.
\begin{lem}\label{shift-continuity}
There is a constant $C$  such that
\begin{equation}\label{whole-shift}
\int_{\RR^2}[\tilde u(\!x+\!rho)-\tilde u(\!x)]^2\mathrm{d}\!x\le C|\!rho|\|u\|^2_{H^1_h}\ \ \forall\ u\in H^1_h.
\end{equation}
\end{lem}
\begin{proof}
Because for an element $\tau\in\T_h$, smooth functions are dense in $H^1(\tau)$, we only need
to prove \eqref{whole-shift}
for functions that are smooth on each element of $\T_h$. Let $u$ be such a piecewise smooth function.
Let $\!rho$ be an arbitrary short vector. We take $\delta=|\!rho|$ and choose a boundary strip $\Omega_\delta$.
The interior part $\Omega^0_{\delta}$ of the domain has the property that
if $\!x\in \Omega^0_{\delta}$ then the line segment $[\!x, \!x+\!rho]\subset \Omega$. We have
\begin{equation*}
\int_{\RR^2}[\tilde u(\!x+\!rho)-\tilde u(\!x)]^2\mathrm{d}\!x
=\int_{\Omega_\delta^0}[u(\!x+\!rho)-u(\!x)]^2\mathrm{d}\!x
+\int_{\RR^2\setminus\Omega_\delta^0}[\tilde u(\!x+\!rho)-\tilde u(\!x)]^2\mathrm{d}\!x.
\end{equation*}
Using Lemma~\ref{bd-shift-theorem}, we bound the second term as
\begin{equation}\label{bd-strip-est}
\int_{\RR^2\setminus\Omega_\delta^0}[\tilde u(\!x+\!rho)-\tilde u(\!x)]^2\mathrm{d}\!x\le
\int_{\Omega_{2\delta}}u^2(\!x)\mathrm{d}\!x\le C|\rho|\|u\|^2_{H^1_h}.
\end{equation}
We then
focus on the first term.
This integral can be taken on an equal measure subset of
$\Omega^0_\delta$. This subset is obtained by removing a zero measure subset that
is composed of such point $\!x$:
$\!x$ or $\!x+\!rho$ is on an open edge $e\in\E^0_h$,
or the closed straight line segment $[\!x, \!x+\!rho]$ connecting $\!x$ and $\!x+\!rho$ contains any vertex of $\T_h$,
or $[\!x, \!x+\!rho]$ overlaps some edges of $\E^0_h$.
By such exclusion, for any $\!x$ in the remaining set, both the ends of $[\!x, \!x+\!rho]$
are in the interior of some open triangular elements, and $[\!x, \!x+\!rho]$ contains no
vertex. The restriction of $u$ on $[\!x, \!x+\!rho]$ is a piecewise smooth one dimensional function,
which may have
a finite number of jumping points in the open straight line segment $(\!x, \!x+\!rho)$.
By the fundamental theorem
of calculus, we have
\begin{equation*}
u(\!x+\!rho)-u(\!x)=\int_0^1\nabla u(\!x+t\!rho)\cdot\!rho \mathrm{d}t+\sum_{p\in
[\!x, \!x+\!rho]\cap\E^0_h}\lbra u\rbra_p
\end{equation*}
Note that the integrand in the integral may make no sense at $t$, if
$\!x+t\!rho\in\E^0_h$, where $u$ may jump. These points are excluded from the integration, where the
jumping effect is resolved in the second term.
On the segment $[\!x, \!x+\!rho]$, $u$ may have a jump at $p\in [\!x, \!x+\!rho]\cap\E^0_h$, which is denoted by
$\lbra u\rbra_p$ that is the value of $u$ from the side of $\!x$ minus that from the side of $\!x+\!rho$.
We thus have
\begin{equation*}
[u(\!x+\!rho)-u(\!x)]^2\le
|\!rho|^2\int_0^1|\nabla u(\!x+t\!rho)|^2\mathrm{d}t+\left[\sum_{p\in[\!x, \!x+\!rho]\cap\E^0_h}\lbra u\rbra_p\right]^2.
\end{equation*}
When we take integral on $\Omega^0_\delta$ (minus the aforementioned zero-measure subset),
the first term is bounded as follows.
\begin{equation*}
\int_{\Omega^0_{\delta}}
|\!rho|^2\int_0^1|\nabla u(\!x+t\!rho)|^2\mathrm{d}t\mathrm{d}\!x=
|\!rho|^2\int_0^1\int_{\Omega^0_{\delta}}|\nabla u(\!x+t\!rho)|^2\mathrm{d}\!x\mathrm{d}t\\
\le
|\!rho|^2\int_{\Omega_h}|\nabla u|^2\mathrm{d}\!x.
\end{equation*}

To estimate the jumping related second term, we write $\lbra u\rbra_p=h_e^{1/2}h_e^{-1/2}\lbra u\rbra_p$
if $p\in [\!x, \!x+\!rho]\cap \E^0_h$ is on the edge $e$, and
use the Cauchy-Schwarz inequality to obtain the following estimate.
\begin{equation*}
\left[\sum_{p\in[\!x, \!x+\!rho]\cap\E^0_h}\lbra u\rbra_p\right]^2\le
\left[\sum_{e\cap[\!x, \!x+\!rho]\ne\emptyset}h_e^{-1}\lbra u\rbra^2_{e\cap[\!x, \!x+\!rho]}\right]
\left[\sum_{e\cap[\!x, \!x+\!rho]\ne\emptyset}h_e\right].
\end{equation*}
We show below that
there is a $C$, depending on the domain $\Omega$ and the
shape regularity $\K$ of $\T_h$ , such that
\begin{equation}\label{zigzag}
\sum_{e\cap[\!x, \!x+\!rho]\ne\emptyset}h_e\le C.
\end{equation}
We then have
\begin{equation*}
\int_{\Omega^0_{\delta}}\left[\sum_{p\in
[\!x, \!x+\!rho]\cap\E^0_h}\lbra u\rbra_p\right]^2\le
C\int_{\Omega^0_{\delta}}
\sum_{e\cap[\!x, \!x+\!rho]\ne\emptyset}h_e^{-1}\lbra u\rbra^2_{e\cap[\!x, \!x+\!rho]}\mathrm{d}\!x.
\end{equation*}
Every term in the right hand side is associated with a particular edge $e\in \E^0_h$. Each edge
$e\in\E^0_h$ is relevant to at most the points
in the parallelogram $\Omega_e$ in the Figure~\ref{shift}.
Thus, by changing the order of sum and integral, we have
\begin{multline*}
\int_{\Omega^0_{\delta}}
\sum_{e\cap[\!x, \!x+\!rho]\ne\emptyset}h_e^{-1}\lbra u\rbra^2_{e\cap[\!x, \!x+\!rho]}\mathrm{d}\!x
\le
\sum_{e\in\E^0_h}h_e^{-1}\int_{\Omega_e}\lbra u\rbra^2_{e\cap[\!x, \!x+\!rho]}\mathrm{d}\!x\\
\le
\sum_{e\in\E^0_h}h_e^{-1}\sin\langle\!rho, e\rangle|\!rho|\int_{e}\lbra u\rbra^2
\le
|\rho|\sum_{e\in\E^0_h}h_e^{-1}\int_{e}\lbra u\rbra^2.
\end{multline*}
Here, $\langle\!rho, e\rangle$ is the angle between the vector $\!rho$ and the edge $e$.
\begin{figure}[!ht]
\centerline{\input{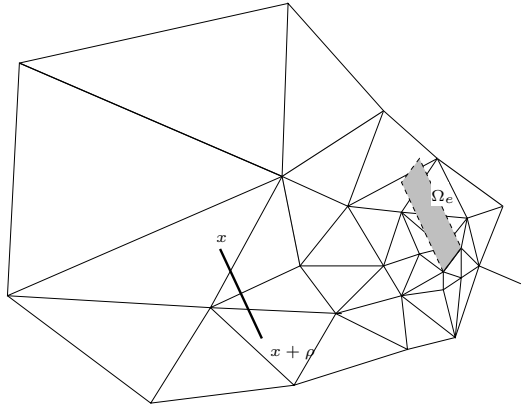}}
\caption{\label{shift}
A $\!rho$ shift and $\Omega_e$ for an edge $e$.}
\end{figure}
Therefore, we have
\begin{equation*} 
\int_{\Omega^0_{\delta}}[u(\!x+\!rho)-u(\!x)]^2d\!x\le
|\!rho|^2|\nabla u|^2_{0, \Omega_h}+|\!rho|\sum_{e\in\E^0_h}h_e^{-1}\int_e\lbra u\rbra^2.
\end{equation*}
Note that the second term may carry the smaller coefficient $|\!rho|\max\{h,|\!rho|\}$
such that the two terms
are closer in the order. But we do not need such refined estimates.
We thus proved
\begin{equation}\label{interior-shift}
\int_{\Omega^0_{\delta}}[u(\!x+\!rho)-u(\!x)]^2\mathrm{d}\!x\le C|\!rho|\|u\|^2_{H^1_h}\ \ \forall\ u\in H^1_h.
\end{equation}
The shift continuity \eqref{whole-shift} then follows from \eqref{interior-shift} and  \eqref{bd-strip-est}.
We have shown that the set of zero extended functions is shift continuous
in $L^2(\RR^2)$.
\end{proof}
We give a proof for the estimate \eqref{zigzag}.
Let $l$ be a straight line cutting through $\Omega$. Let $\T_h$ be a shape regular triangulation
with regularity constant $\K$. Then the sum of lengths of mesh line segments
intersecting $l$ is bounded independent of the triangulation. More specifically,
we prove that there is a constant $C$, depending on the shape regularity $\K$,
but otherwise independent of the triangulation $\T_h$
such that
\begin{equation}\label{zigzag1}
\sum_{e\in\E_h\text{ and }e\cap l\ne\emptyset}h_e\le C.
\end{equation}
We shall use some facts that follow from the shape regularity assumption.
There is a minimum angle $\theta_\K$ among all angles of triangles of $\T_h$. The number of edges
sharing a vertex is bounded by a constant $C$ that only depends on $\K$.
Let $e_1$ and $e_2$ be two edges
sharing a
vertex. There are constants $C_1$ and $C_2$ depending on $\K$ such that
$h_{e_1}\le C_1 h_{e_2}$ and $h_{e_2}\le C_2 h_{e_1}$.
Without loss of generality, we assume $l$ is horizontal. We also
assume that $l$ does not pass any vertex. (This assumption was already made in the context of \eqref{zigzag}.
It can be removed by a slight modification
of the following argument.)
\begin{figure}[!ht]
\centerline{\input{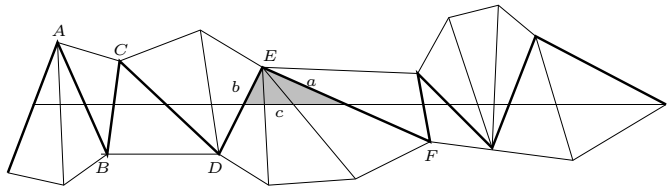}}
\caption{\label{cut-tri}A line cutting through the triangulation.}
\end{figure}
We first trim the set of intersecting edges  $\{e\in\E_h\text{ and }e\cap l\ne\emptyset\}$ to simplify the set without
significantly reducing its sum of lengths of all edges.
Consider the left most edge intersecting $l$,  of which only one end vertex
is shared by some other edges intersecting $l$. Let $A$ be this end vertex, and assume it
is above $l$. We examine all the edges intersecting $l$ and sharing $A$ in the counterclockwise order.
We discard all such edges but the last one that is $AB$ in Figure~\ref{cut-tri}.
(The next edge sharing $A$, as $AC$, does not intersect $l$.) The edge $BC$ must intersect $l$.
There could be other edges intersecting $l$ and sharing the vertex $B$.
We examine all the edges sharing $B$ and intersecting $l$ in the clockwise order. We discard
all but the last one. (It is $BC$ in the figure.) Now the vertex $C$ is in the same situation as $A$, and
we can determine the edge $CD$ using the same rule as for $AB$. Then we determine $DE$, $EF$, and so forth.
We do the trimming all the way
to the right end of $l$. This procedure touches every edge
intersecting $l$, by either trimming an edge off or keeping it. Every edge deleted has at least one vertex-sharing
edge retained.
The remaining edges constitute a continuous piecewise straight path
as represented by the thick line in the figure. We denote this set by $\E^l_h$.
It follows from the aforementioned facts about the shape regular triangulation that there is
a constant $C$, depending on $\K$ only, such that
\begin{equation*} 
\sum_{e\in\E_h\text{ and }e\cap l\ne\emptyset}h_e\le C\sum_{e\in\E^l_h}h_e.
\end{equation*}
We consider a typical triangle bounded by $l$ and $\E^l_h$, as the shaded one
in the figure, whose sides are $a$, $b$, and $c$. Let the angle $\angle DEF$ be denoted by $\theta$.
Then $\theta\ge\theta_\K$. Note that
$c^2=a^2+b^2-2ab\cos\theta$. If $\theta$ is obtuse, then $a+b\le \sqrt 2 c$.
Otherwise, we have $c^2=(a^2+b^2)(1-\cos\theta)+(a-b)^2\cos\theta\le (a^2+b^2)(1-\cos\theta)$. Thus
$a+b\le\sqrt{\frac2{1-\cos\theta}}c$. In any case, we have $a+b\le\sqrt{\frac2{1-\cos\theta_\K}}c$.
We thus proved
\begin{equation*}
\sum_{e\in\E^l_h\cap\E^0_h}h_e\le\sqrt{\frac2{1-\cos\theta_\K}}|l\cap\Omega|.
\end{equation*}
From this, \eqref{zigzag1} follows. Here $|l\cap\Omega|$ is the length of the line segment $l\cap\Omega$ which does exceed the diameter of $\Omega$.
We can now prove the following compact embedding theorem.
\begin{thm}\label{uniformcompactembedding}
Let $\T_{h_i}$ be a (infinite) class of shape regular but not necessarily quasi-uniform
triangulations of the polygonal domain $\Omega$, with a shape regularity constant $\K$.
For each $i$, let $H^1_{h_i}$ be the space of piecewise $H^1$ functions, subordinated to the
triangulation $T_{h_i}$,
equipped with the norm \eqref{H1hnorm}.
Let $\{u_i\}$ be a sequence such that $u_i\in H^1_{h_i}$ for each $i$ and 
there is a constant $C$, such that $\|u_i\|_{H^1_{h_i}}\le C$ for all $i$.
Then, the sequence $\{u_i\}$ has a convergent subsequence in $L^2$.
\end{thm}
\begin{proof}
It follows from \eqref{whole-shift} that the sequence $\{u_i\}$ is a shift-continuous
subset of $L^2$. This, together with the estimate \eqref{tracetheorem}, verifies the
condition for a subset of $L^2$ to be compact, see Theorem~2.12 in \cite{Adams}.
\end{proof}

\section{Discrete Korn's inequality for Naghdi shell}
\label{Korn-Naghdi}
In this section, we prove the discrete Korn's inequality \eqref{N-Korn-discrete} for the Naghdi shell model. 
We let $H_h=\!H^1_h\x \!H^1_h\x H^1_h$, and define a norm in this space by
\begin{equation}\label{h-norm}
\|(\!theta, \!u, w)\|_{H_h}:=\left[\sum_{\alpha=1,2}\left(\|\theta_\alpha\|^2_{H^1_h}+\|u_\alpha\|^2_{H^1_h}\right)+\|w\|^2_{H^1_h}\right]^{1/2}.
\end{equation}
Recall that the norm in $H^1_h$ is defined by \eqref{H1hnorm}.
Let $f(\!theta, \!u, w)$ be a semi-norm that is continuous with respect to the $H_h$  norm such that there is a $C$ and
\begin{equation}\label{N-f-continuous}
|f(\!theta, \!u, w)|\le C\|(\!theta, \!u, w)\|_{H_h}\ \forall\ (\!theta, \!u, w)\in H_h.
\end{equation}
We also assume that $f$ satisfies the condition that
if $(\!theta, \!u, w)\in  \!H^1\x \!H^1\x H^1$ defines a rigid body motion, as explained in the introduction, 
and $f(\!theta, \!u, w)=0$ then $\!theta=0$, $\!u=0$, and $w=0$.

We define a discrete energy norm 
on the space $H_h$ by
\begin{multline}\label{triple-norm}
\|(\!theta, \!u, w)\|_{E_h}:=
\left[\sum_{\alpha,\beta=1,2}\left(\|\rho_{\alpha\beta}(\!theta, \!u, w)\|^2_{0, \Omega_h}+\|\gamma_{\alpha\beta}(\!u, w)\|^2_{0,\Omega_h}\right)+
\sum_{\alpha=1,2}\|\tau_\alpha(\!theta, \!u, w)\|^2_{0,\Omega_h}\right.
\\
\left.+\sum_{e\in \E^0_h}h^{-1}_e\left(
\sum_{\alpha=1,2}\int_{e}\left(\lbra \theta_{\alpha}\rbra^2+\lbra u_{\alpha}\rbra^2\right)
+\int_{e}\lbra w\rbra^2\right)
+f^2(\!theta,\!u, w)\right]^{1/2}.
\end{multline}
The main result of this section is the  following theorem.
\begin{thm}\label{N-Korn-thm}
We assume that $\Omega\subset\RR^2$ is a polygon, on which $\T_h$ is a shape regular but not necessarily quasi-uniform 
triangulation with a regularity constant $\K$. Let $f(\!theta, \!u, w)$ be a seminorm on $H_h$
satisfies the condition \eqref{N-f-continuous}, and the rigid body motion condition.
There exists a constant $C$ that could be dependent on the shell mid-surface
and shape regularity of the triangulation, but otherwise independent of the triangulation,
such that
\begin{equation}\label{N-Korn-inequality}
\|(\!theta, \!u, w)\|_{H_h}\le C\|(\!theta, \!u, w)\|_{E_h}\ \forall\ (\!theta, \!u, w)\in H_h.
\end{equation}
\end{thm}
To prove this theorem, we need a discrete Korn's inequality for plane elasticity for piecewise functions in $H^1_h$, see (1.21) of \cite{Brenner-2}.
It says 
that there is a constant $C$ that might be dependent on the domain $\Omega$ and the shape regularity $\K$ of the triangulation $\T_h$, but
otherwise independent of $\T_h$ such that
\begin{equation}\label{Korn-Brenner}
\sum_{\alpha=1,2}\|u_\alpha\|^2_{H^1_h}\le C\left[\sum_{\alpha=1,2}\|u_\alpha\|^2_{0,\Omega_h}
+\sum_{\alpha, \beta=1,2}\|e_{\alpha\beta}(\!u)\|^2_{0,\Omega_h}+\sum_{e\in\E^0_h}
h^{-1}_e\int_{e}\sum_{\alpha=1,2}\lbra u_{\alpha}\rbra^2\right].
\end{equation}
Here $e_{\alpha\beta}(\!u)=(\partial_\beta u_\alpha+\partial_\alpha u_\beta)/2$ is the plane elasticity strain which is the symmetric part of the gradient of $\!u$.

\begin{proof}[Proof of Theorem~\ref{N-Korn-thm}]
From \eqref{Korn-Brenner},  the definitions \eqref{N-bending}, \eqref{N-metric}, and \eqref{N-shear}
of $\rho_{\alpha\beta}$, $\gamma_{\alpha\beta}$, and $\tau_\alpha$,
and the
definition of the discrete energy norm \eqref{triple-norm}, we see that
there is a constant $C$ such that
\begin{multline}\label{Korn-thm-proof1}
\|(\!theta, \!u, w)\|_{H_h}^2\le C\left[\|(\!theta,\!u, w)\|_{E_h}^2+\sum_{\alpha=1,2}\left(\|\theta_\alpha\|^2_{0,\Omega_h}+\|u_\alpha\|^2_{0, \Omega_h}\right)
+\|w\|^2_{0,\Omega_h}\right]\\
\forall\ (\!theta, \!u, w)\in H_h.
\end{multline}
On a fixed triangulation $\T_h$, it then follows from the Rellich--Kondrachov
compact embedding theorem and Peetre's lemma (Theorem 2.1, page 18 in \cite{Raviart}) that
there is a constant $C_{\T_h}$ such that
\begin{equation*}
\|(\!theta, \!u, w)\|_{H_h}\le C_{\T_h}\|(\!theta,\!u, w)\|_{E_h}\ \ \forall\ (\!theta, \!u, w)\in H_h.
\end{equation*}
We need to show that for a class  of shape regular triangulations, such $C_{\T_h}$ has an upper bound that only
depends on the shape regularity $\K$ of the whole class. 
If this is not true, there would exist a sequence
of triangulations $\{\T_{h_n}\}$ and an associated sequence of functions $(\!theta^n, \!u^n, w^n)$ in
$H_{h_n}$
such that
\begin{equation}\label{N-energy-to-0}
\|(\!theta^n, \!u^n, w^n)\|_{H_{h_n}}=1 \text{ and  }\|(\!theta^n,\!u^n, w^n)\|_{E_{h_n}}\le 1/n.
\end{equation}
It follows from Theorem~\ref{uniformcompactembedding}
that there is a subsequence,
still denoted by $(\!theta^n, \!u^n, w^n)$, such that 
\begin{equation}\label{converge-in-L2}
\lim_{n\to\infty}(\!theta^n, \!u^n, w^n)=(\!theta^0, \!u^0, w^0)\text{ in }\!L^2\x\!L^2\x L^2.
\end{equation}
We show that the limiting functions $\theta^0_\alpha$, $u^0_\alpha$, and $w^0$ are all in $H^1$, 
this limit defines a rigid body motion,  and it is zero, which will lead to a contradiction.

First, we show that $w^0$ is actually in $H^1$ and we have that
$\partial_\alpha w^0+\theta^0_\alpha+b^\beta_\alpha u^0_\beta=0$.
Since $\lim_{n\to\infty}(\!theta^n, \!u^n, w^n)=(\!theta^0, \!u^0, w^0)$ in $\!L^2\x\!L^2\x L^2$, 
in view of the definition \eqref{N-shear}, we have 
\begin{equation*}
\lim_{n\to\infty}\tau_\alpha(\!theta^n, \!u^n, w^n)=\tau_\alpha(\!theta^0, \!u^0, w^0)\ \text{ in } H^{-1}.
\end{equation*}
In the above expressions, the derivatives $\partial_\alpha w^n$ and $\partial_\alpha w^0$ are understood in distributional sense. 

For any compactly supported smooth functions $\phi^\alpha\in\D(\Omega)$, for each $n$, we have
\begin{multline*}
\langle\tau_\alpha(\!theta^n, \!u^n, w^n),\phi^\alpha\rangle
=
-\int_{\Omega_{h_n}}w^n\partial_\alpha\phi^\alpha
+\int_{\Omega_{h_n}}(\theta^n_\alpha+b^\beta_\alpha u^n_\beta)\phi^\alpha
\\=
\int_{\Omega_{h_n}}(\partial_\alpha w^n+\theta^n_\alpha+b^\beta_\alpha u^n_\beta)\phi^\alpha-
\sum_{e\in\E^0_{h_n}}\int_e\lbra w^n\rbra_{n_\alpha}\phi^\alpha.
\end{multline*}
Here, $n_\alpha$ are the components of the unit normal $\!n$ to the edge $e$.
If $e$ is shared by $\tau_1$ and $\tau_2$ with unit outward normals being $\!n_1$ and $\!n_2$, on which the restrictions of 
$w$ are $w_1$ and $w_2$, 
then $\lbra w\rbra_{n_\alpha}=w_1n_{1\alpha}+w_2 n_{2\alpha}$ is the jump of $w$ with respect to $\!n$.
Summation convention is also used in $\lbra w^n\rbra_{n_\alpha}\phi^\alpha$ with $\alpha$ being viewed as a repeated 
sub and super scripts.
Using H\"older inequality and the trace inequality \eqref{trace}, we get
\begin{multline*}
|\langle\tau_\alpha(\!theta^n, \!u^n, w^n),\phi^\alpha\rangle|\le 
\|\tau_\alpha(\!theta^n,\!u^n,w^n)\|_{0,\Omega_{h_n}}\|\phi^\alpha\|_{0,\Omega}\\
+
C\left[\sum_{e\in\E^0_{h_n}}h^{-1}_e\int_e\lbra w^n\rbra^2\right]^{1/2}\sum_{\alpha=1}^2\left[|\phi^\alpha|^2_{0,\Omega}+
\sum_{\tau\in\T_{h_n}}h^2_\tau|\phi^\alpha|^2_{1,\tau}\right]^{1/2}.
\end{multline*}
Since $\|(\!theta^n,\!u^n, w^n)\|_{E_{h_n}}\to 0$, 
we have 
\begin{equation*}
\lim_{n\to\infty}\tau_\alpha(\!theta^n, \!u^n, w^n)=0\ \text{ in } H^{-1}.
\end{equation*}
Therefore, in $H^{-1}$, we have
$\tau_\alpha(\!theta^0, \!u^0, w^0)=\partial_\alpha w^0+\theta^0_\alpha+b^\beta_\alpha u^0_\beta=0$.
Since $\theta^0_\alpha$ and $u^0_\alpha$ are in $L^2$, this equation shows that the weak derivatives of $w^0$ are in $L^2$. Thus we have $w^0\in H^1$ and
\begin{equation}\label{shear=0}
\tau_\alpha(\!theta^0, \!u^0, w^0)=\partial_\alpha w^0+\theta^0_\alpha+b^\beta_\alpha u^0_\beta=0.
\end{equation}

Next we show $u^0_\alpha\in H^1$ and $\gamma_{\alpha\beta}(\!u^0, w^0)=0$.
We have $\lim_{n\to\infty}\gamma_{\alpha\beta}(\!u^n, w^n)=\gamma_{\alpha\beta}(\!u^0, w^0)$ in $H^{-1}$, in which the derivatives 
$\partial_\alpha u^n_\beta$ and $\partial_\alpha u^0_\beta$ are understood in the distributional sense.
Let $\phi^{\alpha\beta}$ be an arbitrary symmetric 
tensor valued function  with components in $\D(\Omega)$. For each $n$ we have
\begin{multline*}
\langle \gamma_{\alpha\beta}(\!u^n, w^n),\phi^{\alpha\beta}\rangle=-\int_{\Omega}u^n_\alpha\partial_\beta\phi^{\alpha\beta}
-\int_{\Omega}(\Gamma^\lambda_{\alpha\beta}u^n_\lambda+b_{\alpha\beta}w^n)\phi^{\alpha\beta}\\
=\int_{\Omega_{h_n}}\gamma_{\alpha\beta}(\!u^n, w^n)\phi^{\alpha\beta}-\sum_{e\in\E^0_{h_n}}\int_e\lbra u^n_{\alpha}\rbra_{n_\beta}\phi^{\alpha\beta}.
\end{multline*}
It follows from this equation and the assumption  
$\lim_{n\to\infty}\|(\!theta^n,\!u^n, w^n)\|_{E_{h_n}}=0$ that we have $\lim_{n\to\infty}\langle \gamma_{\alpha\beta}(\!u^n, w^n),\phi^{\alpha\beta}\rangle=0$.
Thus $\gamma_{\alpha\beta}(\!u^0, w^0)=0$, in which the derivatives are understood in the distributional sense.
In view of the definition \eqref{N-metric}, we have
\begin{equation*}
e_{\alpha\beta}(\!u^0)=\Gamma^\lambda_{\alpha\beta}u^0_\lambda+b_{\alpha\beta}w^0.
\end{equation*}
Since $u^0_\alpha$ and $w^0$ are in $L^2$, so $e_{\alpha\beta}(\!u^0)\in L^2$.
In the sense of distribution, there is the identity that \cite{BCM, Lions}
\begin{equation*}
\partial_{\alpha\beta}u^0_\lambda=\partial_\beta e_{\alpha\lambda}(\!u^0)+\partial_\alpha e_{\lambda\beta}(\!u^0)-\partial_\lambda e_{\alpha\beta}(\!u^0).
\end{equation*}
From this we see that $\partial_{\alpha\beta}u^0_\lambda\in H^{-1}$. It follows from a Lemma of J.L. Lions (whose assumption on the domain 
is met by our polygon, see page 110 of \cite{Lions} and page 
124 of \cite{BCM}) and 
the fact $\partial_{\beta}u^0_\lambda\in H^{-1}$ that
$\partial_{\beta}u^0_\lambda\in L^2$. Therefore, we proved that $u^0_\alpha\in H^1$, and we have  
\begin{equation}\label{membrane=0}
\gamma_{\alpha\beta}(\!u^0, w^0)=0.
\end{equation}

We then show that $\theta^0_\alpha\in H^1$ and $\rho_{\alpha\beta}(\!theta^0, \!u^0, w^0)=0$.
As in the above, in the space $H^{-1}$, we have $\lim_{n\to\infty}\rho_{\alpha\beta}(\!theta^n, \!u^n, w^n)=\rho_{\alpha\beta}(\!theta^0, \!u^0, w^0)$,
in which the derivatives are understood in the distributional sense.
For arbitrary $\phi^{\alpha\beta}\in\D(\Omega)$, and for any $n$, we have the identity  
\begin{multline*}
\langle \rho_{\alpha\beta}(\!theta^n, \!u^n, w^n) , \phi^{\alpha\beta}\rangle=
\int_{\Omega_{h_n}}\rho_{\alpha\beta}(\!theta^n, \!u^n, w^n)\phi^{\alpha\beta}
-\sum_{e\in\E^0_{h_n}}\int_e\lbra \theta^n_{\alpha}\rbra_{n_\beta}\phi^{\alpha\beta}
+\sum_{e\in\E^0_{h_n}}\int_eb^\gamma_\alpha\lbra u^n_{\gamma}\rbra_{n_\beta}\phi^{\alpha\beta}.
\end{multline*}
Using the assumption $\lim_{n\to\infty}\|(\!theta^n,\!u^n, w^n)\|_{E_{h_n}}\to 0$ and this 
equation, we get 
\begin{equation*}
\lim_{n\to\infty}\langle \rho_{\alpha\beta}(\!theta^n, \!u^n, w^n),\phi^{\alpha\beta}\rangle=0.
\end{equation*}
Therefore $\rho_{\alpha\beta}(\!theta^0,\!u^0, w^0)=0$, in which the derivatives $\partial_\alpha\theta^0_\beta$ are  understood in 
in the distributional sense.
In view of the definition \eqref{N-bending} of $\rho_{\alpha\beta}(\!theta^0,\!u^0, w^0)$, 
and the proved regularity that $u^0_\alpha\in H^1$, we see that $e_{\alpha\beta}(\!theta^0)$
are in $L^2$. From this and the argument used above, we see $\partial^2_{\alpha\beta}\theta_\gamma\in H^{-1}$. Using the Lemma of J. L. Lions again,
we get the regularity $\theta^0_\alpha\in H^1$, and the equation
\begin{equation}\label{bending=0}
\rho_{\alpha\beta}(\!theta^0, \!u^0, w^0)=0.
\end{equation}

According to Lemma~3.4 of \cite{BCM},
using the regularities that $\theta^0_\alpha$, $u^0_\alpha$, and $w^0$ are all in $H^1$, and the equations \eqref{shear=0}, \eqref{membrane=0}, and \eqref{bending=0},
we conclude that  the displacement functions 
$(\!theta^0, \!u^0, w^0)$ defines a rigid body motion.

Finally, we show that $(\!theta^0, \!u^0, w^0)=0$.
It follows from the bound \eqref{Korn-thm-proof1}, the assumption \eqref{N-energy-to-0}, and the convergence \eqref{converge-in-L2}
that
$\lim_{n\to\infty}\|(\!theta^n-\!theta^0, \!u^n-\!u^0, w^n-w^0)\|_{H_{h_n}}=0$. Since $f$ is uniformly continuous
with respect to the norm of $H_{{h_n}}$ and since $f(\!theta^n, \!u^n, w^n)\to 0$ ($f$ is a part
of the energy norm \eqref{triple-norm}), we see
$f(\!theta^0, \!u^0, w^0)=0$. Thus $(\!theta^0, \!u^0, w^0)=0$. Therefore,
$\lim_{n\to\infty}\|(\!theta^n, \!u^n, w^n)\|_{H_{h_n}}=0$,
which is contradict to the assumption that
$\|(\!theta^n, \!u^n, w^n)\|_{H_{h_n}}=1$.
\end{proof}


As an example, we take
\begin{equation*}
f(\!theta, \!u, w)=
\left[\sum_{e\in\E^D_h}\int_{e}\sum_{\alpha=1,2}\theta^2_\alpha+\sum_{e\in\E^{S}_h\cup\E^D_h}
\left(\int_{e}\sum_{\alpha=1,2}u^2_{\alpha}+\int_{e}w^2\right)
\right]^{1/2}.
\end{equation*}
It follows from Theorem~\ref{tracetheorem} that there is a $C$ such that
the continuity condition \eqref{N-f-continuous} is satisfied by this $f$.
Under the assumption that  the measure of $\partial^D\Omega$ is positive,
it is verified in \cite{BCM} that if $(\!theta, \!u, w)$ defines a rigid body motion and  $f(\!theta, \!u, w)=0$ then $\!theta=0$, $\!u=0$, and $w=0$.
With this $f$ in the Korn's inequality \eqref{N-Korn-inequality},
we could add boundary penalty term
\begin{equation*}
\sum_{e\in\E^D_h}\int_{e}h^{-1}_e\sum_{\alpha=1,2}\theta^2_\alpha+\sum_{e\in\E^{S}_h\cup\E^D_h}
\left(h^{-1}_e\int_{e}\sum_{\alpha=1,2}u^2_{\alpha}+h^{-1}_e\int_{e}w^2\right)
\end{equation*}
to the squares of both sides of \eqref{N-Korn-inequality} to obtain an inequality that is  useful in  analysis of discontinuous Galerkin 
methods for Naghdi shell in which the essential boundary conditions are enforced by Nitsche's consistent boundary penalty method.

\section{Discrete Korn's inequality for Koiter shell}
\label{Korn-Koiter}
The wellposedness of the Koiter model \eqref{K-model}
is based on the Korn's inequality for Koiter shell that
there is a constant $C$ such that
\begin{multline}\label{Korn}
\sum_{\alpha=1,2}\|u_\alpha\|^2_{H^1}+\|w\|^2_{H^2}\le C
\left[\sum_{\alpha,\beta=1,2}\|\rho^K_{\alpha\beta}(\!u, w)\|^2_{L^2}+\sum_{\alpha,\beta=1,2}\|\gamma_{\alpha\beta}(\!u, w)\|^2_{L^2}+f^2(\!u, w)\right]
\\
\ \ \forall\
\!u\in \!H^1,\ w\in H^2.
\end{multline}
Here $f(\!u, w)$ is a semi-norm on $\!H^1\x H^2$ that satisfies the condition that
if $\!u, w$ defines a rigid body motion, as explained in the introduction, and $f(\!u, w)=0$ then $\!u=0$ and $w=0$.
We  generalize this inequality to piecewise functions on $\Omega_h$.
For Koiter shell problems. 
we also need to use the discrete space $H^2_h$ 
that is composed of piecewise 
$H^2$ functions with the norm defined by \eqref{H2hnorm}.

For $(\!u, w)\in \!H^1_h\x H^2_h$, we define a norm
\begin{equation}\label{K-h-norm}
\|(\!u, w)\|_{H^K_h}=\left(\sum_{\alpha=1,2}\|u_\alpha\|^2_{H^1_h}+\|w\|^2_{H^2_h}\right)^{1/2}.
\end{equation}
Let $f(\!u, w)$ be a semi-norm that is continuous with respect to this norm such that there is a $C$ that only depends on the shell mid-surface
and the regularity $\K$ of the triangulation $\T_h$
\begin{equation}\label{K-f-continuous}
|f(\!u, w)|\le C\|(\!u, w)\|_{H^K_h}\ \forall\ (\!u, w)\in \!H^1_h\x H^2_h.
\end{equation}
We assume that $f$ satisfies the rigid body motion condition that
if $(\!u, w)\in \!H^1\x H^2$ defines a rigid body motion and $f(\!u, w)=0$ then $\!u=0$ and $w=0$.
We define the discrete energy norm on the space $H^K_h$.
\begin{multline}\label{K-triple-norm}
\|(\!u, w)\|_{E^K_h}=
\left[
\sum_{\alpha,\beta=1,2}\|\rho^K_{\alpha\beta}(\!u, w)\|^2_{0,\Omega_h}+\sum_{\alpha,\beta=1,2}\|\gamma_{\alpha\beta}(\!u, w)\|^2_{0,\Omega_h}\right.
\\
\left.+\sum_{e\in \E^0_h}\left(
\sum_{\alpha=1,2}h^{-1}_e\int_{e}\lbra u_{\alpha}\rbra^2
+\sum_{\alpha=1, 2}
h^{-1}_e\int_{e}\lbra \partial_\alpha w\rbra^2
+
h^{-1}_e\int_{e}\lbra w\rbra^2\right)
+f^2(\!u, w)
\right]^{1/2}.
\end{multline}
We then have the following generalization of the Korn's inequality for Koiter shells to piecewise functions.
\begin{thm}\label{K-Korn-thm}
There exists a constant $C$ that could be dependent on the shell mid-surface,
and the shape regularity $\K$ of $\T_h$, but otherwise independent of the triangulation,
such that
\begin{equation}\label{K-Korn-inequality}
\|(\!u, w)\|_{H^K_h}\le C\|(\!u, w)\|_{E^K_h}\ \forall\ w\in H^2_h,\ u_{\alpha}\in H^1_h.
\end{equation}
\end{thm}
\begin{proof}
Koiter's shell model is a restriction of the Naghdi's model on the subspace of zero-shear 
deformations. The discrete Korn's inequality \eqref{K-Korn-inequality} for Koiter shell can be derived from 
the discrete Korn's inequality for Naghdi shell \eqref{N-Korn-inequality}.

For $u_\alpha\in H^1_h$ and $w\in H^2_h$, we define piecewise function $\theta_\alpha\in H^1_h$ by
\begin{equation}
\theta_\alpha=-\partial_\alpha w-b^\gamma_\alpha u_\gamma. 
\end{equation}
We then have $\rho_{\alpha\beta}(\!theta, \!u, w)=-\rho^K_{\alpha\beta}(\!u, w)$ and $\tau_\alpha(\!theta, \!u, w)=0$.
There are constants $C_1$  and $C_2$ only depending on components of the mixed curvature tensor $b^\gamma_\alpha$ 
such that for all $e\in E^0_h$ and for all $\tau\in\T_h$ 
\begin{multline*}
C_1\left(
\sum_{\alpha=1,2}\int_{e}\lbra u_{\alpha}\rbra^2
+\sum_{\alpha=1, 2}
\int_{e}\lbra \partial_\alpha w\rbra^2\right)
\le
\sum_{\alpha=1,2}\int_{e}\lbra u_{\alpha}\rbra^2
+\sum_{\alpha=1, 2}
\int_{e}\lbra \theta_\alpha\rbra^2
\\
\le C_2 \left(
\sum_{\alpha=1,2}\int_{e}\lbra u_{\alpha}\rbra^2
+\sum_{\alpha=1, 2}
\int_{e}\lbra \partial_\alpha w\rbra^2
\right),
\end{multline*}
\begin{multline*}
C_1\left(\sum_{\alpha=1}^2\|u_\alpha\|_{1,\tau}+\|w\|_{2,\tau}\right)\le 
\sum_{\alpha=1}^2\|\theta_\alpha\|_{1,\tau}+\sum_{\alpha=1}^2\|u_\alpha\|_{1,\tau}+\|w\|_{1,\tau}\\
\le C_2
\left(\sum_{\alpha=1}^2\|u_\alpha\|_{1,\tau}+\|w\|_{2,\tau}\right).
\end{multline*}
From these, we get  the equivalences
\begin{equation}
C_1\|(\!u, w)\|_{H^K_h}\le \|(\!theta, \!u, w)\|_{H_h}\le C_2\|(\!u, w)\|_{H^K_h}
\end{equation}
and
\begin{equation}
C_1\|(\!u, w)\|_{E^K_h}\le \|(\!theta, \!u, w)\|_{E_h}\le C_2\|(\!u, w)\|_{E^K_h}.
\end{equation}
The continuity \eqref{K-f-continuous} implies the condition \eqref{N-f-continuous}. 
The inequality \eqref{K-Korn-inequality} then follows from the Korn's inequality for Naghdi shell \eqref{N-Korn-inequality}.
\end{proof}

As an example for the semi-norm satisfying the continuity condition \eqref{K-f-continuous}, we take
\begin{equation*}
f(\!u, w)=
\left[\sum_{e\in\E^{S}_h\cup\E^D_h}
\left(\int_{e}\sum_{\alpha=1,2}u^2_{\alpha}+\int_{e}w^2\right)
+\sum_{e\in\E^D_h}\int_{e}(D_{\!n}w)^2\right]^{1/2}.
\end{equation*}
It follows from Theorem~\ref{tracetheorem} that there is a $C$ only dependent on $\K$ such that
the continuity condition \eqref{K-f-continuous} is satisfied by this $f$.
Under the assumption that  the measure of $\partial^D\Omega$ is positive,
it is verified in \cite{BCM} that if $(\!u, w)\in \!H^1\x H^2$ defines a rigid body motion and $f(\!u, w)=0$ then $\!u=0$ and $w=0$.

\bibliographystyle{plain}

\end{document}